\documentclass{article}

\usepackage[english]{babel}
 \usepackage[letterpaper,top=2cm,bottom=2cm,left=3cm,right=3cm,marginparwidth=1.75cm]{geometry}
\usepackage{bbold}
\usepackage{amsmath}
\usepackage{amsthm}
\usepackage[english]{isodate}
\usepackage[colorlinks=true, allcolors=blue]{hyperref}
\usepackage{bookmark}
\newtheorem{lemma}{Lemma}

\title{Loss Functions for Inventory Control}
\author{Steven R. Pauly \\ Email: s.pauly@slimstock.com}
\date{January 2025}
\begin{document}
\maketitle

\begin{abstract}
    In this paper, we provide analytic expressions for the first-order loss function, the complementary loss function and the second-order loss function for several probability distributions.
    These loss functions are important functions in inventory optimization and other quantitative fields.
    For several reasons, which will become apparent throughout this paper, the implementation of these loss functions prefers the use of an analytic expression, only using standard probability functions.
    However, complete and consistent references of analytic expressions for these loss functions are lacking in literature.
    This paper aims to close this gap and can serve as a reference for researchers, software engineers and practitioners that are concerned with the optimization of a quantitative system.
    This should lead directly to easily using different probability distributions in quantitive models which is at the core of optimization.
    Also, this paper serves as a broad introduction to loss functions and their use in inventory control.
\end{abstract}


\section{Introduction}
\begin{flushleft}
    The first-order loss function, the complementary loss function and the second-order loss function are functions that are used extensively in quantitative fields and especially in inventory control (\cite{Rossi_2014}).
    However, these functions require the calculation of integrals or large sums.
    Analytic expressions circumvent this by expressing the functions in terms of standard probability functions that are implemented efficiently and accurately in all modern software packages.
    Also, if the arguments of these functions are large, numerical accuracy is higher for these analytic expressions (\cite{zipkin2000foundations}).
    That is why these analytic expressions are useful for researchers, software engineers and practitioners in the field of inventory control or other quantitative fields.
\end{flushleft}
\begin{flushleft}
    The main goal of this paper is to provide analytic expressions for loss functions for probability distributions that are commonly encountered in inventory control.
    Although our focus will be on the application to inventory control, the content of this paper can easily be transferred to other quantitative fields.
    In this way, this paper can serve as a reference for any person active in a quantitative field and in need of efficiently and effectively calculating loss functions for several probability distributions.
    We begin in section 2 by explaining these loss functions and their properties.
    In section 3, we look at the state of current literature on these functions and their analytic expressions.
    Then, in section 4, we briefly summarize eight probability distributions that are frequently encountered in inventory control theory.
    We also comment on the distribution's specific use in inventory control theory.
    In section 5, we provide, for every distribution, analytic expressions for the first- and second-order loss function, as well for the complementary loss function.
    We end by giving a brief conclusion and a glimps on possible further research in section 6.
\end{flushleft}
\section{The first-order, second-order and complementary loss function}
\begin{flushleft}
    In this section, we will discuss the mathematical definition of the three loss functions that are the main subject of this paper.
    From their mathematical definition, we can look at how analytic expressions could be found, which is the main goal of the remainder of the paper.
    Let us start with the first-order loss function.
\end{flushleft}
\begin{flushleft}
    The first-order loss function of a continuous probability distribution can be mathematically expressed as
    \begin{equation}
        L_1(r) = \mathbb{E}[[X-r]^+] = \int_r^\infty (x-r)f(x)dx,
    \end{equation}
    where $f(.)$ is the PDF of the distribution (or the PMF when the distribution is discrete).
    By applying the Leibniz integral rule, we can find the derivative of this function.
    The Leibniz integral rule is known to be
    \begin{equation*}
        \frac{d}{dx}\left(\int_{a(x)}^{b(x)}f(x, t)dt\right) = f(x, b(x)) \cdot \frac{d}{dx}b(x)-f(x, a(x)) \cdot \frac{d}{dx}a(x)+\int_{a(x)}^{b(x)}\frac{\partial}{\partial x} f(x, t)dt.
    \end{equation*}
    But as the limits of our function are constants, we can use the special case:
    \begin{equation} \label{eq:2}
        \frac{d}{dx}\left(\int_{a}^{b}f(x, t)dt\right) = \int_{a}^{b}\frac{\partial}{\partial x} f(x, t)dt.
    \end{equation}
    Applying this to the first-order loss function gives us
    \begin{equation}
        \begin{gathered}
            \frac{d}{dr}\left(\int_{r}^{\infty} (x-r)f(x)dx\right)  = \int_{r}^{\infty}\frac{\partial}{\partial r} (x-r)f(x)dx \\
            = \int_{r}^{\infty}-f(x)dx                                                                                         \\
            = F(r)-1,
        \end{gathered}
    \end{equation}
    where $F(\cdot)$ is the CDF of the probability distribution.
    This function is thus non-negative, nonincreasing and convex in $r$.
    Also, it is a well-known result that this function can be rewritten as
    \begin{equation}
        L_1(r) = \int_r^\infty 1-F(x) dx,
    \end{equation}
    which follows from integration by parts.
    The first-order loss function for a discrete distribution is written as
    \begin{equation}
        L_1(r) = \mathbb{E}[[X-r]^+] = \sum_{x \ge r} (x-r)f(x).
    \end{equation}
\end{flushleft}
\begin{flushleft}
    Next to the first-order loss function, we have the complementary loss function.
    This is defined, for the continuous and for the discrete counterpart respectively, as
    \begin{equation}
        L_c(r) = \mathbb{E}[[X-r]^-] = \int_{-\infty}^r (r-x)f(x)dx
    \end{equation}
    and
    \begin{equation}
        L_c(r) = \mathbb{E}[[X-r]^-] = \sum_{x \le r} (r-x)f(x).
    \end{equation}
    Again, by applying the Leibniz integral rule as shown in equation \ref{eq:2}, we can find the derivative of this function to be $F(r)$.
    This function is thus negative, increasing and convex in $r$.
    The complementary loss function can also be written as
    \begin{equation}
        L_c(r) = \int_{-\infty}^{r} F(x) dx,
    \end{equation}
    which is again a result following from integration by parts.
\end{flushleft}
\begin{flushleft}
    Lastly, we have the second-order loss function.
    This function is defined, for a continuous and discrete distribution respectively, as
    \begin{equation}
        L_2(r) = \frac{1}{2}\mathbb{E}[[X-r]^+[X-r-1]^+] = \frac{1}{2}\int_r^\infty (x-r)^2f(x)dx
    \end{equation}
    and
    \begin{equation}
        L_2(r) = \frac{1}{2}\mathbb{E}[[X-r]^+[X-r-1]^+] = \frac{1}{2}\sum_{x \ge r} (x-r)(x-r-1)f(x).
    \end{equation}
    \begin{lemma}
        The second-order loss function $L_2(r)$ can also be written as
        \begin{equation}
            L_2(r) = \int_{r}^{\infty}L_1(x)dx.
        \end{equation}
    \end{lemma}
    \begin{proof}
        $ $\newline
        $ $\newline
        Part I.
        By defining $u(x) = \frac{1}{2}x^2$, $u'(x) = x$, $v(x) = -(1-F(x))$ and $v'(x) = f(x)$, we can use integration by parts to solve $\frac{1}{2}\int_{r}^{\infty}x^2f(x)dx$:
        \begin{equation} \label{eq:12}
            \begin{gathered}
                \frac{1}{2}\int_{r}^{\infty}x^2f(x)dx = \left[-\frac{1}{2}x^2(1-F(x))\right]_{r}^{\infty} + \int_{r}^{\infty}x(1-F(x))dx \\
                = \frac{r^2}{2}(1-F(r))+\int_{r}^{\infty}x(1-F(x))dx.
            \end{gathered}
        \end{equation}
        $ $\newline
        Now, by again applying integration by parts and defining $u(x) = x$, $u'(x) = 1$, $v(x) = -L_1(x)$ and $v'(x) = 1-F(x)$, we can further solve the result of \ref{eq:12}:
        \begin{equation*}
            \begin{gathered}
                \frac{r^2}{2}(1-F(r))+\int_{r}^{\infty}x(1-F(x))dx \\
                = \frac{r^2}{2}(1-F(r))+\left[-xL_1(x)\right]_{r}^{\infty}+\int_{r}^{\infty}L_1(x)dx \\
                = \frac{r^2}{2}(1-F(r))+rL_1(r)+\int_{r}^{\infty}L_1(x)dx.
            \end{gathered}
        \end{equation*}
        $ $\newline
        Part II.
        We can now directly prove our statement by using the result from part I.
        \begin{equation*}
            \begin{gathered}
                \frac{1}{2}\int_{r}^{\infty} (x-r)^2f(x)dx \\
                = \frac{1}{2}\left(\int_{r}^{\infty}x^2f(x)dx - \int_{r}^{\infty}2xrf(x)dx + \int_{r}^{\infty}r^2f(x)dx\right). \\
                = \frac{1}{2}\int_{r}^{\infty}x^2f(x)dx - r\int_{r}^{\infty}xf(x)dx + \frac{r^2}{2}(1-F(r)). \\
                = \frac{1}{2}\int_{r}^{\infty}x^2f(x)dx - r\left[r(1-F(r))+\int_{r}^{\infty}(1-F(x))dx\right] + \frac{r^2}{2}(1-F(r)). \\
                = \frac{1}{2}\int_{r}^{\infty}x^2f(x)dx - r^2(1-F(r))-rL_1(r)+ \frac{r^2}{2}(1-F(r)). \\
                = \frac{1}{2}\int_{r}^{\infty}x^2f(x)dx - \frac{r^2}{2}(1-F(r))-rL_1(r). \\
                = \frac{r^2}{2}(1-F(r))+rL_1(r) + \int_{r}^{\infty}L_1(x)dx- \frac{r^2}{2}(1-F(r))-rL_1(r). \\
                = \int_{r}^{\infty}L_1(x)dx.
            \end{gathered}
        \end{equation*}
    \end{proof}
    Applying equation \ref{eq:2} to this function, gives us
    \begin{equation}
        \begin{gathered}
            \frac{d}{dr}\left(\frac{1}{2}\int_{r}^{\infty} (x-r)f(x)dx\right)  = \frac{1}{2}\int_{r}^{\infty}\frac{\partial}{\partial r} (x-r)^2f(x)dx \\
            = -\int_{r}^{\infty}(x-r)f(x)dx                                                                                         \\
            = -L_1(r).
        \end{gathered}
    \end{equation}
    This function is, just like the first-order loss function; non-negative, nonincreasing and convex in $r$.
    Note that $L_1(\infty) = 0$, $L_c(\infty) = \infty$ and $L_2(\infty) = 0$.
\end{flushleft}
\begin{flushleft}
    Like we already mentioned, in all fields that make use of stochastic modelling, these loss functions are of use (\cite{b07c7eb1-7b50-3adc-9ee9-442be763eb69}).
    For example, in portfolio theory, these loss functions are used frequently (e.g. \cite{Burnecki2005}, \cite{Chen01102011}, \cite{02d5ea92-1a76-39f3-92d6-b5f587e3ae55}, \cite{d729dae9-eb3a-3171-85b1-0f6d6818b172}, \cite{articleSuno}).
    In fields other than inventory control, though, these loss functions are known as so-called partial moments to the $n$-th order with respect to a reference point $r$.
    We can distinct and define a lower and an upper partial moment to the $n$-th order as such:
    \begin{equation*}
        \int_{-\infty}^{r}(r-x)^nf(x)dx,
    \end{equation*}
    \begin{equation*}
        \int_{r}^{\infty}(x-r)^nf(x)dx.
    \end{equation*}
    We can thus look at the first-order and second-order loss as an upper partial moment to the first and second order and the complementary loss function as a lower partial moment to the first order.
\end{flushleft}
\begin{flushleft}
    Let us now manipulate these loss functions to see if we can gain some other insights.
    Let us start with the first-order loss function:
    \begin{equation*}
        \begin{gathered}
            L_1(r) = \int_r^\infty (x-r)f(x)dx \\
            = \int_r^\infty xf(x)dx - r(1-F(r)) \\
            = \int_{-\infty}^{\infty} xf(x)dx - \int_{-\infty}^r xf(x)dx - r(1-F(r)) \\
        \end{gathered}
    \end{equation*}
    \begin{equation} \label{eq:14}
        = \mathbb{E}[X]- \int_{-\infty}^r xf(x)dx - r(1-F(r)).
    \end{equation}
    This result, together with a slight manipulation on the complementary loss function:
    \begin{equation*}
        L_c(r) = \int_{-\infty}^r (r-x)f(x)dx
    \end{equation*}
    \begin{equation}
        = rF(r) - \int_{-\infty}^r xf(x)dx,
    \end{equation}
    gives us a direct relation between the first-order and complementary loss function:
    \begin{equation} \label{eq:16}
        L_1(r) = L_c(r) - r + \mathbb{E}[X].
    \end{equation}
    Also, manipulating the first-order and complementary loss function leads to the insight that the only integral we need to solve, besides the CDF which is readily available in almost all software packages, is the so-called first-order lower partial \textit{raw} moment:
    \begin{equation} \label{eq:17}
        \int_{-\infty}^r xf(x)dx.
    \end{equation}
    We can perform similar manipulations to the second-order loss function:
    \begin{equation} \label{eq:18}
        \begin{gathered}
            L_2(r) = \frac{1}{2}\int_r^\infty (x-r)^2f(x)dx \\
            = \frac{1}{2}\int_r^\infty x^2f(x)dx - r\int_r^\infty xf(x)dx + \frac{r^2}{2}\int_r^\infty f(x)dx \\
            = \frac{1}{2}\int_r^\infty x^2f(x)dx - r\left[\int_{-\infty}^\infty xf(x)dx - \int_{-\infty}^r xf(x)dx\right] + \frac{r^2}{2}(1-F(r)) \\
            = \frac{1}{2}\int_r^\infty x^2f(x)dx - r\mathbb{E}[X] + r\int_{-\infty}^r xf(x)dx + \frac{r^2}{2}(1-F(r)).
        \end{gathered}
    \end{equation}
    $\int_r^\infty x^2f(x)dx$ can be seen as the second-order upper partial \textit{raw} moment.
    This can be split up in the second-order \textit{raw} moment and the second-order lower partial \textit{raw} moment:
    \begin{equation*}
        \int_r^\infty x^2f(x)dx = \int_{-\infty}^\infty x^2f(x)dx - \int_{-\infty}^r x^2f(x)dx,
    \end{equation*}
    where we know that
    \begin{equation*}
        \int_{-\infty}^\infty x^2f(x)dx = Var(X) + \mathbb{E}[X]^2.
    \end{equation*}
    And thus $\int_r^\infty x^2f(x)dx$ can be written as
    \begin{equation*}
        \int_r^\infty x^2f(x)dx = Var(X) + \mathbb{E}[X]^2- \int_{-\infty}^r x^2f(x)dx.
    \end{equation*}
    Eventually, \ref{eq:18} has the following final result:
    \begin{equation} \label{eq:19}
        = \frac{1}{2}\left(Var(X) + \mathbb{E}[X]^2- \int_{-\infty}^r x^2f(x)dx\right) - r\mathbb{E}[X] + r\int_{-\infty}^r xf(x)dx + \frac{r^2}{2}(1-F(r)).
    \end{equation}
    In equation \ref{eq:19}, we have, besides the integral given in expression \ref{eq:17}, only one other integral, which is the second-order counterpart of expression \ref{eq:17}:
    \begin{equation} \label{eq:20}
        \int_{-\infty}^r x^2f(x)dx.
    \end{equation}
\end{flushleft}
\begin{flushleft}
    These loss functions also have easy-to-use analytic expressions.
    These kinds of expressions rely only on the parameters and the probability functions of the probability distribution at hand.
    In essence, we only need to focus on manipulating \ref{eq:17} and \ref{eq:20} to an analytic expression.
    Let us go through an example using a standard normal distribution.
    We start off with expression \ref{eq:17}:
    \begin{equation} \label{eq:21}
        \int_{-\infty}^r xf(x)dx,
    \end{equation}
    where now, $x$ is a random variable that follows a \textit{standard} normal distribution with a mean of $0$ and a standard deviation of $1$ and $f(x)$ is the PDF of the standard normal distribution, evaluated at $x$.
    The PDF of the standard normal distribution is known to be
    \begin{equation*}
        f(x) = \frac{e^\frac{-x^2}{2}}{\sqrt{2\pi}}.
    \end{equation*}
    We can now apply integration by substitution. Defining $u=\frac{-x^2}{2}$ and $\frac{du}{dx} = -x$, we can write expression \ref{eq:21} as
    \begin{equation*}
        -\frac{1}{\sqrt{2\pi}}\int_{-\infty}^r e^udu,
    \end{equation*}
    which is easily solved using the fundamental theorem of calculus. The result is
    \begin{equation} \label{eq:22}
        -\frac{e^{\frac{-r^2}{2}} }{\sqrt{2\pi}}= -f(r).
    \end{equation}
    Now, look at expression \ref{eq:20}.
    To solve this, define $u(x)= x, u'(x)= 1, v(x)=e^{\frac{-x^2}{2}}$ and $v'(x)= -xe^{\frac{-x^2}{2}}$.
    Using integration by parts, we can write
    \begin{equation} \label{eq:23}
        -\int_{-\infty}^r u(x)v'(x)dx = \int_{-\infty}^r e^{\frac{-x^2}{2}} - \left[xe^{\frac{-x^2}{2}}\right]_{-\infty}^r = F(r) - rf(r).
    \end{equation}
\end{flushleft}
\begin{flushleft}
    Using \ref{eq:14}, \ref{eq:16}, \ref{eq:19}, \ref{eq:22} and \ref{eq:23}, we can write the first-order, second-order and complementary loss functions for the standard normal distribution as
    \begin{equation}
        \begin{gathered}
            L_1(r) = \mathbb{E}[X]- \int_{-\infty}^r xf(x)dx - r(1-F(r)) \\
            = -r(1-F(r)) + f(r),
        \end{gathered}
    \end{equation}
    \begin{equation}
        \begin{gathered}
            L_c(r) = rF(r) - \int_{-\infty}^r xf(x)dx \\
            = rF(r) + f(r)
        \end{gathered}
    \end{equation}
    and
    \begin{equation}
        \begin{gathered}
            L_2(r) = \frac{1}{2}\left(Var(X) + \mathbb{E}[X]^2- \int_{-\infty}^r x^2f(x)dx\right) - r\mathbb{E}[X] + r\int_{-\infty}^r xf(x)dx + \frac{r^2}{2}(1-F(r)) \\
            = \frac{1}{2}\left(1 - \int_{-\infty}^r x^2f(x)dx\right) + r\int_{-\infty}^r xf(x)dx + \frac{r^2}{2}(1-F(r)) \\
            = \frac{1}{2}\left(1 - F(r) + rf(r)\right) - rf(r) + \frac{r^2}{2}(1-F(r)) \\
            = \frac{1}{2}\left[(r^2+1)(1-F(r))-rf(r)\right].
        \end{gathered}
    \end{equation}
    It are these analytic expressions for several important probability distributions in inventory control that are the subject of this paper.
    In the upcoming section, we will focus on the current literature of these loss functions and their analytic expressions.
\end{flushleft}
\section{Literature review}
\begin{flushleft}
    The use of these loss functions goes way back. Conceptually, these loss functions go back to the newsvendor problem, first implicitly introduced in \cite{745ab1c9-7630-3ecd-b4b1-574fb6a20b59}.
    If we focus on the literature of inventory control, as we will from now on, one of the first times these loss functions were explicitly mentioned in inventory control theory was in \cite{brown1967decision}.
    From then on, the three loss functions, that were introduced in the previous section, were extensively used in inventory control (\cite{Rossi_2014}, \cite{b07c7eb1-7b50-3adc-9ee9-442be763eb69}).
    Empirical evidence of this are the numerous inventory textbooks that explicitly mention these loss functions (\cite{RePEc:spr:isorms:978-3-319-15729-0}, \cite{brown1967decision}, \cite{nahmias2015production}, \cite{silver1998inventory}, \cite{zipkin2000foundations} is just a small collection of them).
    But why are these loss functions so useful in inventory control?
\end{flushleft}
\begin{flushleft}
    This is because, in inventory control, we try to determine an optimal decision policy with respect to certain performance measures.
    Since the beginning of scientific inventory control, inventory control looked for the structure of these optimal policies (\cite{33fd5e2d-8428-3331-b206-ff93e51b7613}, \cite{Scarf1960}).
    By modelling the inventory system as a Markov Decision Process, researchers quickly found out that in many cases, an optimal policy was governed by just a couple of, often fixed, decision parameters.
    Obviously, the inventory system itself is governed by random variables. To find these optimal decision parameters for a given inventory system, it was needed to evaluate them with respect to the randomness in the system, represented by a probability distribution.
    And if we look at these loss functions; they are all evaluating a function that is governed by a random variable $x$ in function of a fixed parameter $r$.
    These loss functions can thus help in setting up performance meaures and are therefore heavily used in almost all objective functions in all areas of inventory control, which is shown by the large body of literature using them.
    We want to show the reader, via a numerical example, how these loss functions pop up when optimizing an inventory system.
\end{flushleft}
\begin{flushleft}
    We will follow one of the classic approaches for optimizing an arbitrary inventory system and then pick a specific inventory system to complete the example.
    The first step in optimizing an inventory system, and thus a necessary part, is building a stochastic demand model (over the so-called lead time).
    Once this stochastic demand model is set up, we need to find the steady state distribution of the inventory position, based on the chosen ordering policy.
    From this, we can determine the inventory level distribution of the system in steady state.
    The latter distribution can then be used to derive certain performance measures of the inventory system that can eventually be used in optimizing the policy's parameters.
\end{flushleft}
\begin{flushleft}
    We will continue with a concrete example of this process and look at a classic inventory system that is continuously reviewed, the lead time $L$ is fixed and the stochastic demand model is characterized by a continuous distribution.
    In these settings, it is known that an (r, Q) policy is optimal under very general conditions (\cite{ArrowKarlinScarf}). In this policy, we use a reorder point $r$ and whenever our inventory position drops below this reorder point, we order a given quantity $Q$.
    Our task is thus to find the optimal values for $r$ and $Q$.
    We know from this kind of system, that the inventory position in steady state is uniformly distributed on the interval $[r, r+Q]$ (\cite{RePEc:spr:isorms:978-3-319-15729-0}).
    The relationship between the inventory position in steady state at an arbitrary point in time $t$, the demand during the interval $[t, t+L]$ and the inventory level at time $t+L$ is then very simple:
    \begin{equation} \label{eq:27}
        IL(t+L) = IP(t) - DDLT(t+L),
    \end{equation}
    where $IL$ stands for the inventory level, $IP$ stands for the inventory position and $DDLT$ stands for the demand during the lead time.
    Because we know the inventory position distribution, have our demand model and because $t$ is an arbitrary point in time and thus so is $t+L$, we can, from equation \ref{eq:27}, define the steady state distribution of the inventory level in this system as
    \begin{align*}
        P(IL \le x) = \frac{1}{Q}\int_r^{r+Q}\left[1-F(u-x)\right] du,
    \end{align*}
    where $F(\cdot)$ is the CDF of the probability distribution over the lead time interval.
\end{flushleft}
\begin{flushleft}
    We are now ready to calculate some performance measures for this inventory system.
    Two of the most common performance measures are the average stock-out frequency and the average backorders, which can be expressed respectively as
    \begin{equation} \label{eq:28}
        P(IL \le 0) = F(0) = \frac{1}{Q}\int_r^{r+Q}\left[1-F(x)\right] dx
    \end{equation}
    and
    \begin{equation} \label{eq:29}
        \int_{-\infty}^0 F(x) dx = \int_{-\infty}^0 \frac{1}{Q}\int_r^{r+Q}\left[1-F(u-x)\right] dudx.
    \end{equation}
    Now, see that, by using the first-order loss function, equation \ref{eq:28} can be written as
    \begin{equation}
        \frac{1}{Q}\int_r^{r+Q}-L_1'(x) dx
    \end{equation}
    and thus as
    \begin{equation} \label{eq:31}
        \frac{L_1(r)-L_1(r+Q)}{Q}.
    \end{equation}
    Equation \ref{eq:29}, on the other hand, can be directly written as
    \begin{equation} \label{eq:32}
        \int_{-\infty}^0 \frac{1}{Q}\int_r^{r+Q}-L_1'(u-x) dudx.
    \end{equation}
    By changing the order of integration, knowing that $L_1(\infty) = 0$ and using our second-order loss function, we can write expression \ref{eq:32} as
    \begin{equation} \label{eq:33}
        \frac{L_2(r)-L_2(r+Q)}{Q}.
    \end{equation}
    From here, we can use \ref{eq:31} and \ref{eq:33} as equivalents to \ref{eq:28} and \ref{eq:29} and use these in the optimization part.
    This small and stylized example shows that these loss functions are at the very core of optimizing an inventory system.
    This is also the reason why they are used extensively in inventory control.
    In this paper, however, we are primarily focused on the analytic expressions for these loss functions.
    The question that can then be asked is: why should we use these analytic expressions?
\end{flushleft}
\begin{flushleft}
    We see some clear benefits.
    First of all, many researchers use these expressions frequently in calculations and experiments. In these calculations and experiments, they prefer analytic expressions because of both the numerical accuracy and efficiency.
    Evaluating these loss functions requires the calculation of very large sums or integrals and it is almost impossible to get a very high degree of accuracy in a very short time when performing these calculations numerically.
    To avoid numerical inaccracy and inefficiency, these analytic expressions only depend on the probability functions of a distribution which are implemented in almost all commercial software packages, in a fast and accuracte way.
    Also, the use of the cumumative probability function (CDF) can have an extra advantage to researchers and practioners, as the CDF is often used as a performance measure in many situations, especially in practice (and often known as the so-called \textit{Cycle Service Level}).
    These analytic expressions can thus easily translate this \textit{Cycle Service Level} to other service aimed performance measures (e.g. the \textit{Fill Rate}).
\end{flushleft}
\begin{flushleft}
    The biggest benefits, however, will probably come from using these analytic expressions in practice. Let us explain why.
    If we look at the practice of inventory control in combination with these loss functions, we see that almost all analytic expressions of loss functions used in practice are that of the first-order loss function for a (standard) normal distribution.
    One reason that this is the case, is that most performance measures used in practice are simple and straightforward and rely, at best, only on the first-order loss function.
    On the other hand, the normal distribution is mathematically easy to manipulate and in most mainstream resources of inventory control, they primarily do the necessary calculations using the normal distribution as the underlying demand model.
    These reasons lead to the fact that, in practice, the use of these loss functions and also the probability distributions used to represent the underlying demand model are very limited.
    This has led to suboptimization in practice, as it is clear from literature that proper inventory control requires the use of other probability distributions as well as other loss functions (\cite{RePEc:spr:isorms:978-3-319-15729-0}).
\end{flushleft}
\begin{flushleft}
    If we look at inventory control literature and look for analytic expressions of these loss functions, we observe that, for the first-order loss function and a normal distribution, this already popped up regurlarly without explicitly mentioning it as a loss function (e.g. in \cite{hadley1963analysis}).
    Seminal textbooks on inventory control that provide analytic expressions for the first-order loss function for a handful of distributions are, among others, \cite{silver1998inventory} and \cite{zipkin2000foundations}.
    One common pattern throughout these textbooks, however, is that these expressions are limited and somewhat scattershot in the appendix and never nicely collated prominently.
    The resource with the most extensive reference of these analytic expressions is probably \cite{zipkin2000foundations}.
    Here, the author introduces the first- and second-order loss functions and provides analytic expressions for some common, both continuous and discrete, probability distributions.
\end{flushleft}
\begin{flushleft}
    We believe, though, that current literature is not sufficient to help researchers and practioners use these loss functions extensively, as they are supposed to.
    Current literature is, like mentioned, often very limited regardering these analytic expressions. Or in terms of the number of probability distributions provided or in terms of the complete set (first-order, second-order and complementary) of loss functions.
    Also, because these loss functions depend on how the probability functions are defined in terms of the distribution's parameters, the scattering of these expressions is not really helpful, as this will lead to equivalent but different expressions.
    As researchers and practioners are in need of using these loss functions very frequently and often for different probability distributions, they are in need of an easy-to-access reference that is consistent and complete.
    With this paper, we aim to provide exactly that.
\end{flushleft}
\begin{flushleft}
    It is interesting to note that these loss functions have a close relation to many other functions, like, for example, the \textit{mean residual life function} and the \textit{limited expected value function} which are often used also in other areas than inventory control.
    For example, the former is a function from renewal theory that is used extensively in inventory control but also in numerous other engineering applications.
    Analytic expressions for these functions for different probability distributions are thus also implicitly provided with this paper.
    But also the other way around: papers with analytic expressions for these related funtions can give us analytic expressions for our loss functions or at least for some of them.
    When we look at the literature on these related functions and their analytic forms, however, there is, not suprisingly, not much to be found.
    One example though, is \cite{Burnecki2005}, where they provide analytic expressions for the \textit{limited expected value function} $L_e(r)$:
    \begin{equation}
        L_e(r) = \mathbb{E}[min[X, r]] = \int_{-\infty}^{r} xf(x) dx + r(1-F(r)),
    \end{equation}
    in the context of finance (specifically insurance risk assessment). They do this for the following $6$ probability distributions:
    \begin{enumerate}
        \item Log-normal distribution
        \item Exponential distribution
        \item Pareto distribution
        \item Burr distribution
        \item Weibull distribution
        \item Gamma distribution
    \end{enumerate}
    The relation of this function to the first-order loss function is seen easily:
    \begin{equation} \label{eq:35}
        \begin{gathered}
            \int_r^\infty (x-r)f(x)dx \\
            = \int_{r}^{\infty} xf(x) dx - r(1-F(r)) \\
            = \int_{-\infty}^{\infty} xf(x) dx - \int_{-\infty}^{r} xf(x) dx  - r(1-F(r)) \\
            = \mathbb{E}[X] - L_e(r).
        \end{gathered}
    \end{equation}
    This relationship can directly lead to analytic expressions for the first-order loss function for the former mentioned $6$ probability distributions.
    Then, with the relation we gave in equation \ref{eq:16}, we can also easily find the analytic expression for the complementary loss function.
    For example, in \cite{Burnecki2005}, they give the following analytic expression for the limited expected value function for the exponential distribution:
    \begin{equation*}
        L_e(r) = \frac{1}{\beta}(1-e^{-\beta r}),
    \end{equation*}
    with a mean of $\frac{1}{\beta}$.
    Using \ref{eq:35}, we can find the first-order loss function as
    \begin{equation*}
        L_1(r) = \mathbb{E}[X] - L_e(r) = \frac{1}{\beta}-\left[\frac{1}{\beta}(1-e^{-\beta r})\right] = \frac{e^{-\beta r}}{\beta}.
    \end{equation*}
    Subsequently, using equation \ref{eq:16}, we can find the complementary loss function as
    \begin{equation*}
        L_c(r) = L_1(r) + r - \mathbb{E}[X] = \frac{e^{-\beta r}}{\beta} + r - \frac{1}{\beta} = r-\left[\frac{1-e^{-\beta r}}{\beta}\right].
    \end{equation*}
    \cite{Burnecki2005}, however, focusses on probability distributions that are not all very relevant to inventory control and does not explicitly provide the relationship between the limited expected value function and the first-order loss function.
    In this paper, we provide analytic expressions for \textit{all three} loss functions and for \textit{relevant} probability distributions that are frequently encountered in inventory control.
    The next section introduces these different probability distributions and comments on the relevance of them with respect to inventory control.
\end{flushleft}
\section{Common probability distributions in inventory control}
\subsection{Discrete distributions}
\begin{flushleft}
    We will begin by describing four discrete distributions.
    Discrete distributions are particularly useful for inventory control, as demand is nearly always a non-negative integer, i.e. it is a discrete stochastic variable.
    It is therefore natural to use a discrete distribution to model demand, especially when demand is relatively low (\cite{RePEc:spr:isorms:978-3-319-15729-0}).
    \cite{silver1998inventory} uses an empirical observed suitable boundary, to distinguish "high" and "low" demand volume, of $10$ units as the expected mean demand during the lead time.
    Throughout the remainder of this paper, when we are talking about "when demand is relatively low", the reader can take this boundary as a simple, practical definition of it.
\end{flushleft}
\begin{flushleft}
    For every distribution, we give the mean and variance, the range of the discrete random variable X, the parameters and the probability mass function.
    We will also briefly explain why and when this distribution is used in the theory of inventory control.
    We denote the probability mass function as $f(x)$ and is defined on the range.
    The cumulative distribution function is then defined as
    \begin{align*}
        F(x) = \sum_{y\le x}^{} f_y.
    \end{align*}
    If needed, we also provide the \textit{method of moments} for the probability distribution.
\end{flushleft}
\subsubsection{Negative binomial distribution}
\begin{flushleft}
    The negative binomial distribution is widely used in inventory control theory to model the demand over a lead time (\cite{NBN}, \cite{doi:10.1287/mnsc.25.8.777}, \cite{doi:10.1287/opre.33.1.134}, \cite{SYNTETOS200636}, \cite{doi:10.1057/jors.1961.8}) and there is enough empirical evidence that the negative binomial distribution can be used for many items with respect to inventory optimization(\cite{10db68d2-3b29-3c0c-a083-406d6a06c09e}).
    \cite{RePEc:spr:isorms:978-3-319-15729-0} provides a coefficient of dispersion (CD) that can be used to assign a negative binomial distribution to an item, when demand is relatively low:
    \begin{align*}
        CD = \frac{\sigma^2}{\mu} > 1.1,
    \end{align*}
    where $\sigma$ and $\mu$ are (unbiased) estimates of the standard deviation and the mean during the lead time respectively.
    Note that the negative binomial distribution cannot be parametrized when the mean is greater than or equal to the variance.
    This implies that low volume, high erratic items are often suitable to be modelled by a negative binomial distribution.
    We could also deduce this from the fact that the negative binomial distribution is equivalent to a compound Poisson distribution where the compounding distribution is logarithmic.
    Furthermore, random demand with gamma distributed lead times results in negative binomial distributed demand during the lead time (\cite{207bf20a-108d-3b3d-9619-1f323adac508}).
\end{flushleft}
\begin{flushleft}
    The range of the negative binomial distribution is the set of the non-negative integers.
    This distribution has two parameters $n$ and $p$ with $0 < p < 1$ and $n > 0$.
\end{flushleft}
\begin{align*}
    \mathbb{E}[X] = \frac{np}{1-p}
\end{align*}

\begin{align*}
    Var(X) = \frac{np}{(1-p)^2}
\end{align*}

\begin{align*}
    f(x) = {\binom{x+n-1}{n-1}}(1-p)^np^x
\end{align*}

\begin{align*}
    p = 1-\frac{\mathbb{E}[X]}{Var(X)}
\end{align*}

\begin{align*}
    n = \frac{(\mathbb{E}[X])^2}{Var(X)-\mathbb{E}[X]}
\end{align*}

\subsubsection{Geometric distribution}
\begin{flushleft}
    The geometric distribution is, in combination with inventory control, often used  as the compounding distribution in a compound Poisson distribution. This is then known as the Pólya-Aeppli distribution and is used frequently to model demand over a lead time.
    Just as the negative binomial distribution, this compound Poisson distribution is often suitable for items of low volume with highly erratic behavior.
\end{flushleft}
\begin{flushleft}
    The geometric distribution, just as the exponential distribution, has the interesting "no-aging" property, which can simplify the decision-making in inventory control when knowing analytic expressions of the loss functions (\cite{https://doi.org/10.1002/nav.3800230104}).
    Also, if the lead time distribution is geometric, then for any arbitrary distribution of demand per unit with finite cumulants, the lead time demand is asymptotically exponential (\cite{articleCarlson}).
    Therefore, the geometric distribution is often used to model the lead time distribution (\cite{articleCarlson2}).
    Ideally, the lead time demand itself is geometrically distributed. This is, for example, the case when the demand per unit time follows a negative binomial distribution and the lead time is geometrically distributed (\cite{Magistad}).
    When the lead time is exponentially distributed and demand follows a Poisson distribution, and both are independent, their convolution and thus the lead time demand follows a geometric distribution (\cite{articleCarlson2}).
    This is particularly interesting because lead times are often modelled accurately with an exponential distribution and the Poisson distribution is frequently used to model demand processes.
    The geometric distribution is thus often used when we are modelling inventory systems with stochastic lead times.
    Also, the geometric distribution is commonly used to model the time until some event occurs, such as the failure of a machine (as the probability of failure is constant, regardless of the age of the machine) (\cite{zipkin2000foundations}).
    Lastly, it is interesting to point out that the geometric distribution is a special case of the negative binomial distribution with the number of failures $n$ being equal to $1$.
\end{flushleft}

\begin{flushleft}
    The range of the geometric distribution is the set of the positive integers.
    This distribution has one parameter $p$ with $0 < p < 1$.
\end{flushleft}

\begin{align*}
    \mathbb{E}[X] = \frac{1}{p}
\end{align*}

\begin{align*}
    Var(X) = \frac{1-p}{p^2}
\end{align*}

\begin{align*}
    f(x) = (1-p)^{x-1}p
\end{align*}

\begin{align*}
    p = \frac{1}{\mathbb{E}[X]}
\end{align*}

\subsubsection{Logarithmic distribution}
Just as the geometric distribution, the logarithmic distribution is often used to model the demand size in a compound Poisson distribution.
Actually, a compound Poisson distribution where the demand size is logaritmically distributed, is a negative binomial distribution.
The range is the set of the positive integers and the distribution has only one parameter $p$ with $0 < p < 1$.

\begin{align*}
    \mathbb{E}[X] = -\frac{p}{(1-p)\ln(1-p)}
\end{align*}

\begin{align*}
    Var(X) = -\frac{p(\ln(1-p)+p)}{(1-p)^2\left(\ln(1-p)\right)^2}
\end{align*}

\begin{align*}
    f(x) = -\frac{p^x}{x\ln(1-p)}
\end{align*}

\begin{align*}
    p = 1-e^{W_{-1}\left(-\frac{1}{\mathbb{E}[X]e^{\frac{1}{\mathbb{E}[X]}}}\right)+\frac{1}{\mathbb{E}[X]}},
\end{align*}
where $W_k$ is the product logarithm.

\subsubsection{Poisson distribution}
\begin{flushleft}
    The Poisson distribution is a well-known and widely used distribution in the theory of inventory control.
    It is used extensively in modelling the demand over a fixed lead time.
    Demands occur one unit at a time and in every small time interval, a demand may or may not occur.
    This leads to the Poisson distribution being a suitable demand model when demand is of low volume and not too erratic.
    \cite{RePEc:spr:isorms:978-3-319-15729-0} provides a coefficient of dispersion (CD) that can be used to assign a Poisson distribution to an item, when demand is relatively low:
    \begin{align*}
        CD = 0.9 \le \frac{\sigma^2}{\mu} \le 1.1.
    \end{align*}
\end{flushleft}
\begin{flushleft}
    Besides the modelling of the demand size, it is very often used to model the customer order intensity.
    Together with another distribution that models the demand size, like the logarithmic or geometric distribution, it then forms a compound Poisson distribution.
    Like already stated, these compound Poisson distributions are often used to accurately model erratic lead time demand.
\end{flushleft}
The range of the Poisson distribution is the set of the non-negative integers. The distribution has one parameter $\lambda$ with $\lambda > 0$.

\begin{align*}
    \mathbb{E}[X] = Var(X) = \lambda
\end{align*}

\begin{align*}
    f(x) = \frac{\lambda^xe^{-\lambda}}{x!}
\end{align*}

\subsection{Continuous distributions}
\begin{flushleft}
    We now describe four continuous distributions.
    Continuous distributions are useful for inventory control when the demand during the lead time is relatively high (\cite{RePEc:spr:isorms:978-3-319-15729-0}).
    Like already stated, \cite{silver1998inventory} uses an empirical observed suitable boundary, to distinguish "high" and "low" demand volume, of $10$ units as the expected mean demand during the lead time.
    Throughout the remainder of this paper, when we are talking about "when demand is relatively high", the reader can take this boundary as a simple, practical definition of it.
\end{flushleft}
\begin{flushleft}
    For every distribution, we again give the mean and variance, the range of the continuous random variable $X$, the parameters and the probability density function.
    We will again briefly explain why and when this distribution is used in the theory of inventory control.
    We denote the probability density function as $f(x)$ and is defined on the range.
    The cumulative distribution function is then defined as

    \begin{align*}
        F(x) = \int_{-\infty}^x f(x) dx.
    \end{align*}

    Again, if needed, we provide the \textit{method of moments} for the probability distribution.
\end{flushleft}

\subsubsection{Normal distribution}
\begin{flushleft}
    The normal distribution is probably the most used distribution to model lead time demand in inventory control when we do not know the exact distribution or this does not really matter a lot.
    One of the main reasons behind this fact is the \textit{Central Limit Theorem (CLT)}.
    This has led it to being a popular distribution, together with its mathematical convenience and wide availability in software.
    The problem with the normal distribution, however, is that there is always a small probability for negative demand.
    It is therefore best to approximate items' demand model with a normal distribution, only when demand is relatively high and the coefficient of variation (CV) is not \textit{too} high.
    Most specialized resources use a limit of $CV \le 0.5$.
\end{flushleft}
The range for the normal distribution is the real line and it has two parameters $\mu$ and $\sigma$ with
$\mu \in \mathbb{R}$ and $\sigma \in \mathbb{R}_{>0}$.

\begin{align*}
    \mathbb{E}[X] = \mu
\end{align*}

\begin{align*}
    Var(X) = \sigma^2
\end{align*}

\begin{align*}
    f(x) = \frac{1}{\sigma\sqrt{2\pi}}e^{-\frac{1}{2}\left(\frac{x-\mu}{\sigma}\right)^2}
\end{align*}

\subsubsection{Gamma distribution}
\begin{flushleft}
    The gamma distribution is used extensively in practice to model demand over a lead time when demand is relatively high (\cite{8285761c-e9d6-323a-bbd4-de8efa04aa5d}, \cite{512d921f-2e73-3182-8129-e175ce204488}, \cite{https://doi.org/10.1002/nav.3800300216}, \cite{SNYDER1984373}, \cite{TIJMS1984175}, \cite{articleTyworth}) and is popular mostly because of the fact that there is no density on negative numbers and therefore more appropriate for inventory control applications (as negative demand is normally not probable).
    This has led to the fact that many demand processes at companies were well aproximated by a gamma distribution, especially when dealing with relatively high demand and a coefficient of variation that was not considerably less than $1$.
\end{flushleft}
\begin{flushleft}
    The gamma distribution is the continuous analogue of the negative binomial distribution and just like the negative binomial distribution, it is a generalization of another distribution (in this case, the exponential distribution).
    Furthermore, the demand of perishable items can often be modelled with a gamma distribution (\cite{BROEKMEULEN20093013}, \cite{d729dae9-eb3a-3171-85b1-0f6d6818b172}).
\end{flushleft}
\begin{flushleft}
    The range of the gamma distribution contains the non-negative real numbers.
    This distribution has two parameters $\alpha$ and $\beta$ with $\alpha, \beta > 0$.

    \begin{align*}
        \mathbb{E}[X] = \frac{\alpha}{\beta}
    \end{align*}

    \begin{align*}
        Var(X) = \frac{\alpha}{\beta^2}
    \end{align*}

    \begin{align*}
        f(x) = \frac{\beta\left(\beta x\right)^{\alpha-1}e^{-\beta x}}{\Gamma(\alpha)}
    \end{align*}

    \begin{align*}
        \alpha = \frac{(\mathbb{E}[X])^2}{Var(X)}
    \end{align*}

    \begin{align*}
        \beta = \frac{\mathbb{E}[X]}{Var(X)}
    \end{align*}

\end{flushleft}

\subsubsection{Log-normal distribution}
\begin{flushleft}
    Althoug the log-normal distribution is not extensively used in inventory control theory, it has been found extremely valuable for analyzing statistical data arising from different fields like agriculture, economics, medicine, and so on (\cite{https://doi.org/10.2307/1235218}).
    In the area of inventory control, the log-normal distribution is often used to model the demand distribution over a time unit for a group of items (\cite{brown1959statistical}, \cite{DAS1983267}, \cite{Herron01061976}, \cite{https://doi.org/10.1111/j.1540-5915.1972.tb00547.x}, \cite{Howard1976}).
    However, the log-normal distribution is also used to model lead time demand on item level (\cite{COBB20131842}, \cite{DAS1983267}, \cite{TADIKAMALLA1979553}).
    This is mostly motivated by (1) the empirical observation that demand for an item during the delivery lead time can often be represented by a log-normal distribution (\cite{Trux1971}) and (2) the log-normal distribution has some interesting properties:
    \begin{enumerate}
        \item The log-normal distribution represents a strictly positive random variable with a lower bound, if any.
        \item It is positively skewed and can assume many shapes by changing a single parameter.
        \item It reduces the skewness of raw data to a large extent.
        \item Powers, products and multiples of log-normal distributed variables are also log-normal.
    \end{enumerate}
\end{flushleft}
\begin{flushleft}
    The range contains the positive real numbers.
    This distribution has two parameters $\mu$ and $\sigma$ with $\mu \in (-\infty, +\infty)$ and $\sigma > 0$.

    \begin{align*}
        \mathbb{E}[X] = e^{\mu+\frac{\sigma^2}{2}}
    \end{align*}

    \begin{align*}
        Var(X) = [e^{\sigma^2}-1]e^{2\mu+\sigma^2}
    \end{align*}

    \begin{align*}
        f(x) = \frac{1}{x\sigma\sqrt{2\pi}}e^{-\frac{1}{2}\left(\frac{\ln(x)-\mu}{\sigma}\right)^2}
    \end{align*}

    \begin{align*}
        \mu = \log\left(\frac{(\mathbb{E}[X])^2}{\sqrt{(\mathbb{E}[X])^2+Var(X)}}\right)
    \end{align*}

    \begin{align*}
        \sigma = \sqrt{\log\left(1+\frac{Var(X)}{(\mathbb{E}[X])^2}\right)}
    \end{align*}

\end{flushleft}

\subsubsection{Exponential distribution}
\begin{flushleft}
    The exponential distribution has shown to be very helpful in connection with inventory control (\cite{brown1977materials}).
    The exponential distribution is often used to model the distribution of a stochastic lead time (\cite{articleCarlson2}, \cite{https://doi.org/10.1002/nav.3800230104}) but it is also a popular distribution to model demand for perishable items (\cite{doi:10.1080/17509653.2022.2134222}) and hazardous items (\cite{10.3844/jssp.2009.183.187}).
    In general, an exponential distribution is often used for deteriorating items (\cite{articleOphokenshi}) but also for slow moving items that are continuous in nature (\cite{SNYDER1984373}).
\end{flushleft}
\begin{flushleft}
    Besides that, just like the geometric distribution, its discrete analogue, it can simplify decision-making significantly (\cite{ArrowKarlinScarf}) and is therefore a popular distribution to use in the mathematical study of inventory control.
    The range of the exponential distribution contains the non-negative real numbers and it has one parameter $\beta$ with $\beta > 0$.

    \begin{align*}
        \mathbb{E}[X] = \frac{1}{\beta}
    \end{align*}

    \begin{align*}
        Var(X) = \frac{1}{\beta^2}
    \end{align*}

    \begin{align*}
        f(x) = \beta e^{-\beta x }
    \end{align*}

    \begin{align*}
        \beta = \frac{1}{\mathbb{E}[X]}
    \end{align*}

\end{flushleft}

\section{Analytic expressions for the first-order, second-order and complementary loss function}
\subsection{Discrete distributions}
We will now provide analytic expressions for the loss functions for the discrete distributions.
When having a discrete random variable $X$, the loss functions can be defined as follows.

\begin{enumerate}
    \item First-order loss function:

          \begin{align*}
              L_1(r) = \mathbb{E}[[X-r]^+] = \sum_{x \ge r} (x-r)f(x)
          \end{align*}

    \item Complementary loss function:

          \begin{align*}
              L_c(r) = \mathbb{E}[[X-r]^-] = \sum_{x \le r} (r-x)f(x)
          \end{align*}

    \item Second-order loss function:

          \begin{align*}
              L_2(r) = \frac{1}{2}\mathbb{E}[[X-r]^+[X-r-1]^+] = \frac{1}{2}\sum_{x \ge r} (x-r)(x-r-1)f(x)
          \end{align*}

\end{enumerate}

\subsubsection{Negative binomial distribution}

Let $F_0(x)$ represent the cumulative distribution function (CDF) of the negative binomial distribution, evaluated at $x$ and with parameters $n$ and $p$.
Then, let $F_1(x)$ represent the CDF of the negative binomial distribution, evaluated at $x$ and with parameters $n+1$ and $p$.
And finally, let $F_2 (x)$ represent the CDF of the negative binomial distribution, evaluated at $x$ and with parameters $n+2$ and $p$.

\begin{equation}
    L_1(r) = \frac{np}{1-p}[1-F_1(r-2)]-r[1-F_0(r-1)]
\end{equation}

\begin{equation}
    L_c(r) = rF_0(r-1)-\frac{np}{1-p}F_1(r-2)
\end{equation}

\begin{equation}
    L_2(r) = \left(\frac{r^2+r}{2}\right)[1-F_0(r-1)]-\frac{rnp}{1-p}[1-F_1(r-2)]+\frac{(np)^2+np^2}{2(1-p)^2}[1-F_2(r-3)]
\end{equation}

\subsubsection{Geometric distribution}

\begin{equation}
    L_1(r) = \frac{(1-p)^r}{p}
\end{equation}

\begin{equation}
    L_c(r) = \frac{(1-p)^r+pr-1}{p}
\end{equation}

\begin{equation}
    L_2(r) = \frac{(1-p)^{r+1}}{p^2}
\end{equation}

\subsubsection{Logarithmic distribution}

Define $\beta = -\frac{1}{\ln(1-p)}$.

\begin{equation}
    L_1(r) = \frac{\beta p^r}{1-p}-r[1-F(r-1)]
\end{equation}

\begin{equation}
    L_c(r) = rF(r) - \beta\left[\frac{1-p^{r+1}}{1-p}-1\right]
\end{equation}

\begin{equation}
    L_2(r) = \frac{1}{2}[r^2+r][1-F(r-1)]+\frac{\beta(2r+1)p^r}{2(p-1)}-\frac{\beta p^r[p(r-1)-r]}{2(1-p)^2}
\end{equation}

\subsubsection{Poisson distribution}

\begin{equation}
    L_1(r) = -(r-\lambda)[1-F(r)]+\lambda f(r)
\end{equation}

\begin{equation}
    L_c(r) = (r-\lambda)F(r)+\lambda f(r)
\end{equation}

\begin{equation}
    L_2(r) = \frac{1}{2}\left([(r-\lambda)^2+r][1-F(r)]-\lambda(r-\lambda)f(r)\right)
\end{equation}

\subsection{Continuous distributions}
We will now provide analytic expressions for the loss functions for the continuous distributions.
When having a continuous random variable $X$, the loss functions can be defined as follows.

\begin{enumerate}

    \item First-order loss function:

          \begin{align*}
              L_1(r) = \mathbb{E}[[X-r]^+] = \int_r^\infty (x-r)f(x)dx
          \end{align*}

    \item Complementary loss function:

          \begin{align*}
              L_c(r) = \mathbb{E}[[X-r]^-] = \int_{-\infty}^r (r-x)f(x)dx
          \end{align*}

    \item Second-order loss function:

          \begin{align*}
              L_2(r) = \frac{1}{2}\mathbb{E}[[X-r]^+[X-r-1]^+] = \frac{1}{2}\int_r^\infty (x-r)^2f(x)dx
          \end{align*}

\end{enumerate}

\subsubsection{Normal distribution}

Let $p = \frac{r-\mu}{\sigma}$. $F(p)$ is the CDF of the standard normal distribution (i.e. a normal distribution with mean and variance of $0$ and $1$ respectively) and $f(p)$ is the probability density function of the standard normal distribution.

\begin{equation}
    L_1(r) = (\mu-r)[1-F(p)]+\sigma f(p)
\end{equation}

\begin{equation}
    L_c(r) = (r-\mu)F(p)+\sigma f(p)
\end{equation}

\begin{equation}
    L_2(r) = \frac{1}{2}[(r-\mu)^2+\sigma^2][1-F(p)]-\frac{\sigma}{2}f(p)[r-\mu]
\end{equation}

\subsubsection{Gamma distribution}

Let $F_0(x)$ represent the CDF of the gamma distribution, evaluated at $x$ and with parameters $\alpha$ and $\beta$.
Then, let $F_1(x)$ represent the CDF of the gamma distribution, evaluated at $x$ and with parameters $\alpha+1$ and $\beta$.
And finally, let $F_2(x)$ represent the CDF of the gamma distribution, evaluated at $x$ and with parameters $\alpha+2$ and $\beta$.

\begin{equation}
    L_1(r) = \frac{\alpha}{\beta}[1-F_1(r)]-r[1-F_0(r)]
\end{equation}

\begin{equation}
    L_c(r) = rF_0(r)-\frac{\alpha}{\beta}F_1(r)
\end{equation}

\begin{equation}
    L_2(r) = \frac{r^2}{2}[1-F_0(r)]-\frac{r\alpha}{\beta}[1-F_1(r)]+\frac{\alpha(\alpha+1)}{2\beta^2 }[1-F_2(r)]
\end{equation}

\subsubsection{Log-normal distribution}

Let $F(p_i)$ be the CDF of the standard normal distribution and define:
$p_1 = \frac{\ln(r)-\mu-2\sigma^2}{\sigma},$
$p_2 = \frac{\ln(r)-\mu-\sigma^2}{\sigma}$
and
$p_3 = \frac{\ln(r)-\mu}{\sigma}.$

\begin{equation}
    L_1(r) = e^{\mu+\frac{\sigma^2}{2}}[1-F(p_2)]-r[1-F(p_3)]
\end{equation}

\begin{equation}
    L_c(r) = rF(p_3)-e^{\mu+\frac{\sigma^2}{2}}F(p_2)
\end{equation}

\begin{equation}
    L_2(r) = \frac{r^2}{2}[1-F(p_3)]-re^{\mu+\frac{\sigma^2}{2}}[1-F(p_2)]+\frac{e^{2\left(\mu+\sigma^2\right)}}{2}[1-F(p_1)]
\end{equation}

\subsubsection{Exponential distribution}

\begin{equation}
    L_1(r) = \frac{e^{-\beta r}}{\beta}
\end{equation}

\begin{equation}
    L_c(r) = r-\left[\frac{1-e^{-\beta r}}{\beta}\right]
\end{equation}

\begin{equation}
    L_2(r) = \frac{e^{-\beta r}}{\beta^2}
\end{equation}

\section{Conclusion and further research}
\begin{flushleft}
    In this paper, we have provided analytic expressions for three very important so-called loss functions.
    These loss functions are at the core of the calculations needed to perform inventory optimization.
    As these loss functions are depending on a specific probability distribution, we looked for the most important and widely used distributions in inventory control.
    This has led to a complete and consistent reference, available for researchers, practioners and software developers, which was lacking in current inventory theory literature.
    We believe that both practice and theory will greatly benefit from these analytic expressions.
\end{flushleft}
\begin{flushleft}
    This paper has the goal to enable researchers and practioners to easily use these loss functions in calculations for and simulations of real-life inventory systems.
    One step further would to be find easy-to-use closed-form approximations of these loss functions.
    For the normal distribution and the first-order loss function, some work has already been done here (e.g. \cite{DESCHRIJVER20121375}, \cite{Rossi_2014}, \cite{200fda1f-f795-3a50-b422-4d66e1cfba96}).
    However, inventory theory could benefit from simple closed-form approximations of all loss functions for all distributions mentioned in this paper. This would open new possibilities in finding closed-form solutions for complex inventory models.
\end{flushleft}
\begin{flushleft}
    We need to keep in mind that, especially in inventory control, simplicity is key to connect theory and practice.
    By not aiming for simplicity, we run the risk of developing a large pool of solutions which are not used in practice and thus not really add value to the companies we are implicitly adressing.
\end{flushleft}
\bibliographystyle{plain}
\bibliography{LossFunctionsForInventoryControl.bib}

\end{document}